\documentclass[oneside, 16pt]{amsart}
\usepackage[utf8]{inputenc}

\usepackage{amssymb}
\usepackage{mathtools}
\usepackage{ mathrsfs }
\usepackage{tikz-cd}
\usepackage[hidelinks]{hyperref}
\usepackage{enumitem}

\usepackage{natbib}
\setcitestyle{numbers,square}
\newcommand{\BZ}{\mathbb{Z}}
\newcommand{\BQ}{\mathbb{Q}}
\newcommand{\BN}{\mathbb{N}}
\newcommand{\BR}{\mathbb{R}}
\newcommand{\BC}{\mathbb{C}}
\newcommand{\BF}[1]{\mathbb{F}_{#1}}
\newcommand{\BG}[0]{\mathbb{G}}
\newcommand{\BA}[0]{\mathbb{A}}

\newcommand{\BT}[0]{\mathbb{T}}

\newcommand{\Qbarl}{\overline{\BQ}_{\ell}}

\newcommand{\CF}[0]{\mathcal{F}}
\newcommand{\CM}[0]{\mathcal{M}}

\newcommand{\SH}{\mathscr{H}}
\newcommand{\SL}[0]{\mathscr{L}}
\newcommand{\SF}[0]{\mathscr{F}}

\newcommand{\Hyp}[4]{\textnormal{Hyp}(#1, #2; #3; #4)}

\newcommand{\Spec}[1]{\textnormal{Spec}(#1)}

\newtheorem{theorem}{Theorem}[section]
\newtheorem*{theorem*}{Theorem}
\newtheorem{corollary}[theorem]{Corollary}
\newtheorem{proposition}[theorem]{Proposition}
\newtheorem{lemma}[theorem]{Lemma}
\theoremstyle{definition}
\newtheorem{definition}[theorem]{Definition}
\newtheorem{remark}[theorem]{Remark}

\title{Hypergeometric sheaves with tannakian monodromy group $G_2$}
\author[]{Beat Zurbuchen}

\begin{document}
	\begin{abstract}
		\hspace{-6.5pt} Based on a suggestion by Katz, we determine the Tannakian monodromy group of certain $\ell$-adic hypergeometric sheaves to be the exceptional group $G_2$. Using the smoothing properties of the Fourier transform over the integers, known as uniformity theorems, we prove that the fourth moment is constant on an open locus of the family of hypergeometric sheaves. In our example, this implies a comparison theorem for the Tannakian monodromy groups which determines these groups if the characteristic is large.
	\end{abstract}
	\maketitle
	\tableofcontents
	
	\section{Introduction} The equidistribution of Frobenius conjugacy classes has been a cornerstone of number theory at least since Dirichlet's theorem on primes in arithmetic progressions. In a geometric context, this study was successfully initiated by Deligne in \citep{DeligneWeilII}. The central object of this theory is a compact Lie group, the so-called monodromy group, in which Frobenius conjugacy classes acting on a local system on a variety over a finite field naturally occur and then equidistribute. This provides a formalism to derive concrete equidistribution results in large generality. These results can be applied to local systems whose Frobenius traces parametrize a family of exponential sums. Such an application implies equidistribution results for the family of exponential sums.
	
	These monodromy groups were determined for a large class of $\ell$-adic local systems, among them the so-called hypergeometric sheaves (see \citep[Ch.~8]{KatzESDE}), by Katz. For example, consider a finite field $k$ of odd characteristic $p$ and a prime $\ell \neq p$. Let $\psi$ be a non-trivial $\ell$-adic additive character of $k$. Let $\lambda_2$ denote the unique non-trivial $\ell$-adic multiplicative character of order 2 of $k$ and denote the associated Gauss sums by
	\[
	A_{\psi, k} := -\sum_{x \in k^*}\lambda_2(x)\psi(x).
	\]  
	For each $a \in k^*$ and each $\ell$-adic multiplicative character $\chi$ of $k$ define
	\[
	\text{Hyp}(a, \chi) := A_{\psi, k}^{-7}\sum_{\substack{x_1\cdots x_7 = x_8a\\ x_1, \ldots, x_8 \in k^*}} \psi(x_1 + \ldots + x_7 - x_8)\chi\big(x_4x_5(x_6x_7)^{-1}\big)\lambda_2(x_8).
	\]
	Define the complex of constructible $\ell$-adic sheaves on $\BG_{m, k}$ (see \citep[{}8.2.2]{KatzESDE} for the notation)
	\[
	\CF(\chi, \psi, k) := (A_{\psi, k}^{-7})^{\text{deg}}\otimes\Hyp{!}{\psi}{1, 1, 1, \chi, \chi, \overline{\chi}, \overline{\chi}}{\lambda_2}
	\]
	where $ (A_{\psi, k}^{-7})^{\text{deg}}$ denotes the pullback along the structure map $\BG_{m, k} \rightarrow \Spec{k}$ of the $\ell$-adic character which maps $\text{Fr}_k \mapsto A_{\psi, k}^{-7}$. This complex has the property
	\[
	\text{Tr}(\text{Fr}_k|\CF(\chi, \psi, k)_{{a}}) = \text{Hyp}(a, \chi)
	\]
	for all $a \in k^*$ and all multiplicative characters $\chi.$ Relations such as this, which are consequences of the Lefschetz fixed point formula for constructible $\ell$-adic sheaves, are the primary reason for the appearance of cohomological methods in the theory of exponential sums.
	
	Fix a multiplicative character $\chi$ of $k$ and consider the algebraic group $G_2$ over $\Qbarl$. Fix an isomorphism $\tau\colon \Qbarl \cong \BC$. This isomorphism allows us to base change $G_2$ to $\BC$ to obtain a complex algebraic group $G_{2, \BC}$. Pick a maximal compact subgroup $UG_2 \subset G_{2, \BC}(\BC)$ and let $UG_2^\natural$ be the space of conjugacy classes in $UG_2$. Following \citep[Thm.~8.1]{KatzG2}, there are semisimple conjugacy classes $\theta_{a, \chi} \in UG_2^\natural$ such that the trace of $\theta_{a,\chi}$ acting on the unique irreducible seven-dimensional representation of $UG_2$ is given by
	\[
	\text{Tr}(\theta_{a, \chi}) = \tau(\text{Hyp}(a, \chi)).
	\]
	Denote by $k_n$ a finite extension of $k$ of degree $n \geq 1$ and for a multiplicative character $\chi$ on $k$ define $\chi_n(x) := \chi(\text{Nm}_{k_n/k}(x))$ for all $x \in k_n^*$. When $p \geq 17,$ Deligne's equidistribution theorem \citep[Thm.~7.11.1]{KatzESDE} applied to the determination of the monodromy group of $\CF(\chi, \psi, k)$ in \citep[Thm.~8.1]{KatzG2} implies
	\[
	\lim_{n \rightarrow \infty} \frac{1}{|k_n^*|}\sum_{a \in k_n^*} f(\theta_{a, \chi_n}) = \int_{UG_2} f(g)dg
	\]
	for any continuous, central function $f \in C(UG_2)$ where $dg$ denotes the probability Haar measure on $UG_2.$ We say that the set $\{\theta_{a, \chi} : a \in k_n^*\}$, where $\chi$ remains fixed, equidistributes in $UG_2^\natural$ for the pushforward of the probability Haar measure from $UG_2$ as $n \rightarrow \infty$.
	
	The equidistribution result by Deligne is not able to describe the distribution of the set of conjugacy classes $\{\theta_{a, \chi}: \chi \in \widehat{k_n^*}\}$ for fixed $a$ and variable $\chi$ because the multiplicative characters are not representable by a variety. In \citep[Ch.~4]{KatzConvEqui}, a subcategory of the category of $\ell$-adic perverse sheaves on $\BG_{m, k}$ is equipped with the structure of a neutral $\Qbarl$-linear Tannakian category, which associates an arithmetic Tannakian monodromy group to each object in this category. By design, this formalism proves and describes the equidistribution of such sets. Katz explains in the book \citep[Ch.~25]{KatzConvEqui} (see also Theorem \ref{THM_Katz}) how to apply this formalism to the above sums and determines the Tannakian monodromy group for most of the perverse sheaves occurring in the family defined by the variable $a$.
	\begin{theorem*}[\hspace*{-0mm}{\citep[Thm. 25.1]{KatzConvEqui}}]\label{THM_KatzWeirdVersion}
		The set of $a \in k^*$ such that the set of conjugacy classes $\{\theta_{a, \chi} : \chi \in \widehat{k_n^*}\}$ equidistributes in $UG_2^\natural$ as $n \rightarrow \infty$ has cardinality $|k^*| + o(|k^*|)$ as  $|k| \rightarrow \infty$.
	\end{theorem*}
	Based on a suggestion in \citep[Rmk. 25.8]{KatzConvEqui}, we improve this theorem to the following result.
	\begin{theorem*}[Theorem \ref{THM_MainResult}]
		Suppose the characteristic of $k$ is large enough and let $a \in k^*$. The set of conjugacy classes $\{\theta_{a, \chi} : \chi \in \widehat{k^*_n}\}$ equidistributes in $UG_2^\natural$ as $n \rightarrow \infty$.
	\end{theorem*}
	We are not able to use our method to give an effective lower bound on the characteristic of $k$. The reason lies in the application of Theorem \ref{THM_KatzUniformity}. This theorem provides an ineffective non-empty open subset of the spectrum of the integers such that for each finite field $k$ whose spectrum maps into the open subset the fourth moment is constant in the family defined by $a \in k^*$. We do not know how to provide such an open subspace without appealing to Theorem \ref{THM_KatzUniformity}, thus we cannot provide an effective lower bound on the characteristic. For an exposition of uniformity results, such as Theorem \ref{THM_KatzUniformity}, the reader is referred to \citep[Ch.~V]{KiehlWeissPervSheafFourier}.
	
	Furthermore, it is not clear whether the Tannakian monodromy group associated with hypergeometric sheaves remains $G_2$ over fields of small characteristic. For monodromy groups of hypergeometric sums, it is a well-known phenomenon (see \citep[Thm~8.1]{KatzG2} and \citep[Thm.~14.10]{KatzESDE}) that (geometric) monodromy groups become uniform only for fields of large characteristic. For example, it is proven in \citep[Thm.~8.1]{KatzG2} that when the characteristic satisfies $p < 17$ the monodromy group of $\CF(\chi, \psi, k)$ becomes a finite group contained in $G_2$ for some multiplicative characters $\chi$. The Tannakian monodromy group in our example could degenerate to the smaller group $\textnormal{SL}_2$ (see Theorem \ref{THM_Katz}). On the other hand, Katz demonstrated in a small number of cases -- via computer calculations -- that the Tannakian monodromy group is $G_2$ (see \citep[Rem.~25.8]{KatzConvEqui}).
	
	Our method should extend to more general hypergeometric families, such as those constructed in \citep{GabberLoeserTore}. Consider $n > 1$ and a perverse !-hypergeometric sheaf $\SH$ on $\BG_{m, k}^n$ in the sense of \citep[Def.~8.1.2]{GabberLoeserTore}. In favorable cases, a group morphism $\BG_m^n \rightarrow \BG_m$ induces a hypergeometric family by restricting $\SH$ to the fibers. The perverse sheaves studied in this paper are of this form up to negligible composition factors -- that is, factors whose Euler-Poincar\'{e} characteristic vanishes. In future work, we plan to apply this method to determine the Tannakian monodromy groups for other members of such families.
	
	Our proof relies heavily on Katz's determination of the Tannakian monodromy group for a generic member in the family parametrized by $a$. We study the sum $\textnormal{Hyp}(a, \chi)$, where $a \in k^*$. Lemma \ref{LEM_ChangeOfPsi} shows that varying the base point $a$ is equivalent to varying the additive character $\psi$. The key step is to formulate the fourth moment of the Tannakian monodromy group as a weighted Euler-Poincar\'{e} characteristic (see Theorem \ref{THM_MomentEulerPoincare}). We then show that this fourth moment -- and hence the Tannakian monodromy group itself -- is independent of $\psi$, if the characteristic is large enough, by appealing to the uniformity result of Theorem \ref{THM_KatzUniformity}. This implies that the Tannakian monodromy group is independent of the parameter $a$, provided the characteristic of $k$ is large enough. In particular, the Tannakian monodromy group for any parameter $a$ has to agree with the Tannakian monodromy group of a generic member, which is $G_2$.

	\begin{center}
		\textbf{Notation: }
	\end{center}
	\begin{itemize}
		\item $\ell$: a fixed prime.
		\item All sheaves and complexes of $\ell$-adic sheaves on a separated, noetherian scheme $X$ on which $\ell$ is invertible are objects in the category $D_c^b(X, \Qbarl)$ defined in \citep[1.1.2]{DeligneWeilII}. If $X$ is of finite type over $\BZ$, then we say a complex of $\ell$-adic sheaves on $X$ is mixed if its cohomology sheaves are mixed in the sense of \citep[1.2.2]{DeligneWeilII}.
		\item We fix an isomorphism $\Qbarl \cong \BC$ once and for all; we apply this isomorphism implicitly whenever needed. In particular, the notation $\lim_{n \rightarrow \infty}$ denotes a limit of complex numbers.
		\item $k$: a finite field, the characteristic of $k$ is always coprime to $\ell$ and \textit{odd}. We fix an algebraic closure $\overline{k}$.
		\item $k_n$: the unique subfield $k_n \subset \overline{k}$ such that $k_n/k$ has degree $n$.
		\item $\SL_\eta$: if $\eta$ is a character $\eta\colon G(k) \rightarrow \Qbarl^*$ for some finite field $k$ and an algebraic group $G/k$, then $\SL_\eta$ denotes the local system on $G$ obtained from $\eta$ using the Lang torsor construction.
		\item $\chi$: a multiplicative character $\chi\colon k^* \rightarrow \Qbarl^*$; we extend any such character to a character on any finite extension $k_n/k$ by putting $\chi_n(x) := \chi(\text{Nm}_{k_n/k}(x))$ for all $x \in k_n$.
		\item $\lambda_2:$ the unique non-trivial multiplicative character of order 2 on $k$ for any finite field $k$.
		\item $\psi$: an additive character $\psi\colon k^+ \rightarrow \Qbarl^*$; we extend any such character to any finite extension $k_n/k$ by putting $\psi_n(x) := \psi(\text{Tr}_{k_n/k}(x))$ for all $x \in k_n$.
		\item $\text{Hyp}(-, -; -; -)$: a hypergeometric sheaf in the sense of \citep[8.2.2]{KatzESDE}.
		\item $M_{2m}(M)$: if $M$ is an object in a Tannakian category, we define the $2m$-th moment by the formula $M_{2m}(M) := \text{dim}(\text{Hom}(1, M^{\otimes m}\otimes(M^\vee)^{\otimes m}))$, where $\text{Hom}$ denotes the vector space of morphisms in the Tannakian category and $1$ is the tensor unit.
		\item If $M$ is a bounded complex of $\Qbarl[T]$-modules, such that all its cohomology sheaves are finite-dimensional over $\Qbarl$, then we define $\text{Tr}(T|M) := \sum_{i \in \BZ}(-1)^i \text{Tr}(T|H^i(M))$. For example, if $\SF$ is an $\ell$-adic sheaf on a separated $k$-scheme $X$, then $\text{Tr}(\text{Fr}_k|H^\bullet_c(X, \SF)) = \sum_{i \geq 0} (-1)^i\text{Tr}(\text{Fr}_k|H^i_c(X, \SF))$.
	\end{itemize}
	\section{The hypergeometric complexes} We recall certain definitions and results from \citep{KatzConvEqui}. Let $k$ be a finite field, $\psi$ a non-trivial additive character of $k$, and $a \in k^*$. Denote by $\SH_{\psi, k, a}$ the $\ell$-adic constructible complex on $\BG_{m, k}$ which is designated $N(a, k)$ in the book \citep[Ch.~27, p.~165]{KatzConvEqui}. We recall the construction of $\SH_{\psi, k, a}$. Define the $\ell$-adic complexes on $\BG_{m, k}\times \BG_{m, k}$
	\begin{align*}
		\CM_1 &:= [(z, t) \mapsto a/(zt^2)]^*\text{Hyp}(!, \psi; 1, 1, 1; \lambda_2)\\ \CM_2 &:= [(z, t) \mapsto z]^*\text{Hyp}(!, \psi; 1, 1; \emptyset) \\ \CM_3 &:= [(z, t) \mapsto zt]^*\text{Hyp}(!, \psi; 1, 1; \emptyset)
	\end{align*}
	where $1$ denotes the trivial multiplicative character of $k$. Consider $$N_0(a, k) := (A_{\psi, k}^{-7})^{\text{deg}}\otimes R\pi_{2!}(\CM_1\otimes\CM_2\otimes\CM_3)[2]$$ where $\pi_2(z, t) := t$. It is shown in \citep[Lem.~27.1~(1)]{KatzConvEqui} that this complex is perverse. Note that $N_0(a, k)$ is mixed of weight $\leq 0$ by \citep[Lem.~27.2]{KatzConvEqui} and \citep[Var.~6.2.3]{DeligneWeilII}. We define the perverse sheaf $\SH_{\psi, k, a}$ to be the graded piece of the weight filtration of $N_0(a, k)$ which is pure of weight zero. Note that $\SH_{\psi, k, a}$ is a quotient of $N_0(a, k)$ because $N_0(a, k)$ is of weight $\leq 0$.
	\begin{remark}
		The complex $N_0(a, k)$ is constructed precisely so that 
		\[
		{H}^{\bullet}_c(\BG_{m, \overline{k}}, N_0(a, k)\otimes\mathscr{L}_{\chi}) \cong \CF(\chi, \psi_n, k_n)_a
		\]
		for all multiplicative characters $\chi$ of $k_n$. 
	\end{remark}
	
	\begin{definition}[{\citep[p.~21]{KatzConvEqui}}]
		Let $j\colon \BG_{m, \overline{k}} \rightarrow \mathbb{P}^1_{\overline{k}}$ be the inclusion. A multiplicative character $\chi$ of $k_n$ is called \textit{good} (or \textit{not bad}) for a perverse sheaf $M$ on $\BG_{m, \overline{k}}$ if the natural morphism
		\[
		Rj_!(M\otimes\SL_\chi) \rightarrow Rj_*(M\otimes\SL_\chi)
		\]
		is an isomorphism. 
	\end{definition}
	\begin{lemma}\label{LEM_GoodChar}
		Let $M$ be a perverse sheaf on $\BG_{m, k}$ which is pure of weight $w \in \BZ$ and $\chi$ a good character for $M$. The cohomology groups
		\[
		H^\bullet_c(\BG_{m, \overline{k}}, M\otimes\SL_\chi) 
		\]
		are concentrated in degree zero and pure of weight $w$.
	\end{lemma}
	\begin{proof}
		The Leray spectral sequence implies that the map
		\[
		{H}_c^\bullet(\BG_{m, \overline{k}}, M\otimes \SL_\chi) \rightarrow {H}^\bullet(\BG_{m, \overline{k}}, M\otimes \SL_\chi)
		\]
		is an isomorphism. Since the complex $M\otimes \SL_\chi$ is perverse, Artin's vanishing theorem \citep[Thm.~III.6.1]{KiehlWeissPervSheafFourier} says that the cohomology groups on the left are concentrated in non-negative degrees and the cohomology groups on the right are concentrated in non-positive degrees since $\BG_{m, k}$ is affine. Thus they are concentrated in degree zero. Note that \citep[Var.~6.2.3]{DeligneWeilII} says that the left side is of weight $\leq w$ and the right side is of weight $\geq w$. Thus the cohomology group is pure of weight $w$.
	\end{proof}
	\begin{remark}
		The vanishing is recorded in \citep[p.~21]{KatzConvEqui}. The purity of the cohomology group follows from  \citep[Thm.~4.1~(3bis)]{KatzConvEqui} applied to $N = M\otimes\SL_\chi$. In fact, Theorem 4.1 in in loc. cit. states states that a character is good for a perverse sheaf if and only if the cohomological Mellin coefficient is pure.
	\end{remark}
	\begin{theorem}\label{THM_KatzHypProp}
		The complex $\SH_{\psi, k, a}$ is pure of weight zero, perverse, irreducible, has no bad characters, and there is a Frobenius-equivariant isomorphism
		\[
		{H}^{\bullet}_c(\BG_{m, \overline{k}}, \SH_{\psi, k, a}\otimes\mathscr{L}_{\chi}) \cong \CF(\chi, \psi_n, k_n)_a
		\]
		for each multiplicative character $\chi \neq \lambda_2$ of $k_n$. 
	\end{theorem}
	\begin{proof}
		This complex is perverse and pure of weight zero by construction. It has no bad characters by \citep[Lem. 27.5]{KatzConvEqui} and is irreducible by \citep[Lem.~27.11]{KatzConvEqui}. The cohomology groups are determined in \citep[Lem.~27.4]{KatzConvEqui}.
	\end{proof}
	\begin{remark}
		The character $\lambda_2$ does not occur in the above statement because the graded pieces of the weight filtration of $N_0(a, k)$ which are not pure of weight zero are geometrically isomorphic to $\SL_{\lambda_2}$ by the argument in \citep[Lem.~27.3]{KatzConvEqui}. Note that the character $\lambda_2$ is self-dual, so the quotient map $N_0(a, k) \rightarrow \SH_{\psi, k, a}$ induces an isomorphism
		\[
		H^\bullet_c(\BG_{m, \overline{k}}, N_0(a, k)\otimes\SL_\chi) \rightarrow H^\bullet_c(\BG_{m, \overline{k}}, \SH_{\psi, k, a}\otimes\SL_\chi) 
		\]
		for a multiplicative character $\chi$ of $k_n$ if and only if $\chi \neq \lambda_2$. 
	\end{remark}
	The formalism of \citep[Ch. 4]{KatzConvEqui} (see also \citep[Sec.~4.2]{FresanNickWork}) equips the category of perverse sheaves on $\BG_{m, k}$ which have no perverse subquotients with vanishing Euler-Poincar\'{e}  characteristic when pulled back to $\overline{k}$ with the structure of a neutral Tannakian category. The Tannakian formalism associates to any object $M$ in this category an algebraic group over $\Qbarl$, which is defined as the automorphism group of a fiber functor of the Tannakian category generated by $M$. This algebraic group is called the \textit{arithmetic Tannakian monodromy group} of $M$. The perverse sheaves on $\BG_{m, \overline{k}}$ which have no perverse subquotients with vanishing Euler-Poincar\'{e}  characteristic can also be equipped with the structure of a neutral Tannakian category (see \citep[Ch.~2]{KatzConvEqui}). This associates to each object $M$ in this category the \textit{geometric Tannakian monodromy group}, which is the automorphism group of a fiber functor of the Tannakian category generated by $M$. If the perverse sheaf on $\BG_{m, k}$ is geometrically semisimple, the geometric Tannakian monodromy group is a normal subgroup of the arithmetic Tannakian monodromy group by \citep[Thm.~6.1]{KatzConvEqui}. When the inclusion of the geometric Tannakian monodromy group into the arithmetic Tannakian monodromy group is an equality, we will refer to the common group as the \textit{Tannakian monodromy group}.
	
	As stated in the introduction, Katz determines the Tannakian monodromy group of a generic member of the hypergeometric family $\SH_{\psi, k, a}$, in the family parametrized by $a$, in \citep[Ch.~25]{KatzConvEqui}. Moreover, it is shown that the corresponding monodromy group is either $\textnormal{SL}_2$ or $G_2.$ These facts are summarized in the following theorem.
	\begin{theorem}\label{THM_Katz} Let $\psi, k, a$ be as above. 
		\begin{enumerate}
			\item The geometric and the arithmetic Tannakian monodromy group of $\SH_{\psi, k, a}$ agree. The Tannakian monodromy group of this complex is a subgroup of $\textnormal{GL}_7.$
			\item The Tannakian monodromy group of $\SH_{\psi, k, a}$ in $\textnormal{GL}_7$ is conjugate to either the image of the unique irreducible 7-dimensional representation of $G_2$  or the image of the $\textnormal{SL}_2$-representation $\textnormal{Sym}^6(\textnormal{std}_2)$ where $\textnormal{std}_2$ denotes the representation defined by the inclusion $\textnormal{SL}_2 \rightarrow \textnormal{GL}_2$.
			\item There exists $n \geq 1$ and $b \in k_n^*$ such that the monodromy group of $\SH_{\psi_n, k_n, b}$ is $G_2$. 
			\item The fourth moment satisfies $M_4(\SH_{\psi, k, a}) = 4$ if and only if the Tannakian monodromy group of $\SH_{\psi, k, a}$ is $G_2.$
		\end{enumerate}
	\end{theorem}
	\begin{proof}
		The points 1. and 2. are proven in \citep[Lem.~25.2]{KatzConvEqui}. The point 3. follows from \citep[Thm.~25.1]{KatzConvEqui}. To prove 4., the moment is $4$ when the Tannakian monodromy group is $G_2$ and the moment is $7$ if the Tannakian monodromy group is $\textnormal{SL}_2$. These moments were computed using the program LiE. 
	\end{proof}
	\section{The fourth moment} In the previous section, we proved that it is sufficient to compute the fourth moment of $\SH_{\psi, k, a}$ to determine its Tannakian monodromy group. The goal of this section is to derive a concrete formula for the fourth moment as a limit using the equidistribution results for Tannakian monodromy groups.
	\begin{definition}\label{DEF_FunctionF}
		Define the Laurent polynomial $$P(x_{i, j}, y_{i, j}) := \prod_{j \in \{1, 2\}} x_{4, j}x_{5, j}y_{6, j}y_{7, j}(y_{4, j}y_{5, j}x_{6, j}x_{7, j})^{-1}$$
		as an element of the ring $$P \in \BZ[(x^{\pm 1}_{i, j})_{(i, j) \in  \{1, \ldots, 8\}\times\{1, 2\}}, (y_{i, j}^{\pm 1})_{(i, j) \in \{1, \ldots, 8\}\times\{1, 2\}}].$$
		Let $k$ be a finite field, $\psi$ a non-trivial additive character of $k$, and $a \in k^*$. Define
		\begin{align*}
			f(\psi, k, a) := |k|^{-15}\sum_{\substack{x_{1, j}\dots x_{7, j} = ax_{8, j} \\ y_{1, j} \dots y_{7, j} = ay_{8, j}\\P(x_{i, j}, y_{i, j}) = 1}} \Bigg( &\psi\Bigg(\sum_{j \in \{1, 2\}} \Bigg(\sum_{i = 1}^7 (x_{i, j} - y_{i, j} )-x_{8, j}+ y_{8, j} \Bigg)\Bigg) \\ & \times \lambda_2\Bigg(\prod_{j \in \{1, 2\}} x_{8, j}y_{8, j}^{-1}\Bigg)\Bigg)
		\end{align*}
		where we sum over all values $x_{i, j} \in k^*$ and $y_{i, j} \in k^*$ which satisfy the condition in the sum.
	\end{definition}
	\begin{theorem}\label{THM_EquiForMom}
		Let $k$ be a finite field, $\psi$ a non-trivial additive character of $k$, and $a \in k^*$. We have
		\[
		M_4(\SH_{\psi, k, a}) = \lim_{n \rightarrow \infty} f(\psi_n, k_n, a).
		\]
	\end{theorem}
	\begin{proof}
		% Let $K$ be a maximal compact subgroup in the Tannkain monodromy group of $\SH_{\psi, k, a}$ and $\rho\colon K \rightarrow \GL_7$ the inclusion of the. The Peter-Weyl theorem implies
		% \[
		% M_4(\SH_{\psi, k, a}) = \int_{K} |\text{Tr}(\rho)(g)|^4 dg
		% \]
		% where $\rho\colon K \rightarrow \GL_n$ is the representation associated to $\SC_{\psi, k, a}.$
		By Theorem \ref{THM_KatzHypProp}, all characters are good for $\SH_{\psi, k, a}$. Thus \citep[Thm.~7.3]{KatzESDE} and \citep[Eqn. 9.2]{KowalskiTannaka} imply
		\[  
		M_4(\SH_{\psi, k, a}) = \lim_{n \rightarrow \infty} \frac{1}{|k_n|- 1}\sum_{\chi \in \widehat{k_n^*}} |\text{Tr}(\text{Fr}_{k_n}|H^{\bullet}_c(\BG_{m, \overline{k}}, \SH_{\psi, k, a}\otimes\mathscr{L}_{\chi}))|^4.
		\]
		Remark that the terms in this sum are represented by hypergeometric sums
		\begin{align*}
		\text{Tr}(\text{Fr}_{k_n}&|H^{\bullet}_c(\BG_{m, \overline{k}}, \SH_{\psi, k, a}\otimes\mathscr{L}_{\chi})) =\\ &(A_{\psi, k})^{-7n}\sum_{\substack{x_1\cdots x_7 = ax_8\\ x_1, \ldots, x_8 \in k_n^*}} \psi_n(x_1 + \ldots + x_7 - x_8)\chi(x_4x_5(x_6x_7)^{-1})\lambda_2(x_8)
		\end{align*}
		for all multiplicative characters $\chi \neq \lambda_2$ of $k_n$ by Theorem \ref{THM_KatzHypProp} and \citep[{}8.2.7]{KatzESDE}. A priori, we can not simply replace the terms in the above formula for the moment by hypergeometric sums because of the term with $\chi = \lambda_2$. We replace the terms in this sum by hypergeometric sums by first excluding the term with $\chi = \lambda_2$. Afterwards, we replace the terms with hypergeometric sums, and then we complete the sum by adding the hypergeometric sum with $\chi = \lambda_2$.

		The complex $\SH_{\psi, k, a}$ is pure of weight zero. Lemma \ref{LEM_GoodChar} implies that the cohomology groups $H^{\bullet}_c(\BG_{m, \overline{k}}, \SH_{\psi, k, a}\otimes\mathscr{L}_{\lambda_2})$ are pure of weight zero and concentrated in degree zero. This implies that the term
		\[
		|\text{Tr}(\text{Fr}_{k_n}|H^{\bullet}_c(\BG_{m, \overline{k}}, \SH_{\psi, k, a}\otimes\mathscr{L}_{\lambda_2}))|^4 \leq \chi_c(\BG_{m, \overline{k}}, \SH_{\psi, k, a}))^4
		\]
		is bounded as $n \rightarrow \infty$. Hence we can exclude the character $\lambda_2$ from the limit and write
		\begin{align*}
			M_4(\SH_{\psi, k, a}) &= \lim_{n \rightarrow \infty} \frac{1}{|k_n|- 1}\sum_{\chi \neq \lambda_2} |\text{Tr}(\text{Fr}_{k_n}|H^{\bullet}_c(\BG_{m, \overline{k}}, \SH_{\psi, k, a}\otimes\mathscr{L}_{\chi}))|^4.
		\end{align*}
		We apply the expression of the terms as hypergeometric sums to obtain
		\begin{align*}
			M_4(&\SH_{\psi, k, a}) = \lim_{n \rightarrow \infty}\frac{|A_{\psi, k}|^{-28n}}{|k_n|- 1}\sum_{\chi \neq \lambda_2}\sum_{\substack{x_{1, j}\dots x_{7, j} = ax_{8, j} \\ y_{1, j} \dots y_{7, j} = ay_{8, j}}} \Bigg(\chi(P(x_{i,j}, y_{i, j})
			% \prod_{j \in \{1, 2\}} x_{4, j}x_{5, j}y_{6, j}y_{7, j}(y_{4, j}y_{5, j}x_{6, j}x_{7, j})^{-1}
			)\\&\times \lambda_2\Bigg(\prod_{j \in \{1, 2\}} x_{8, j}(y_{8, j})^{-1}\Bigg)\psi_n\Bigg(\sum_{j \in \{1, 2\}} \Bigg(\sum_{i = 1}^7 (x_{i, j} - y_{i, j} )-x_{8, j}+ y_{8, j}\Bigg)\Bigg)\Bigg).
		\end{align*}
		It follows from the cancellation theorem \citep[{}8.4.7]{KatzESDE} that we have an exact sequence of perverse sheaves
		\begin{align*}
		0 \rightarrow V\otimes\SL_{\lambda_2}[1] \rightarrow &\text{Hyp}(!, \psi; 1, 1, 1, \lambda_2, \lambda_2, \lambda_2, \lambda_2; \lambda_2)  \\\rightarrow  &\text{Hyp}(!, \psi; 1, 1, 1, \lambda_2, \lambda_2, \lambda_2; \emptyset) \rightarrow 0
		\end{align*}
		where $V := H_c^0(\BG_{m, \overline{k}}, \text{Hyp}(!, \psi; \lambda_2, \lambda_2, \lambda_2, 1, 1, 1; \emptyset))$. Note that $V$ is pure of weight 6 by K\"{u}nneth's formula and $ \text{Hyp}(!, \psi: 1, 1, 1, \lambda_2, \lambda_2, \lambda_2; \emptyset)$ is pure of weight 6 by \citep[Thm.~8.4.2~(4)]{KatzESDE}. By \citep[Thm.~8.4.2~(6)]{KatzESDE}, we get the estimate
		\[
		\Bigg|(A_{\psi, k})^{-7}\sum_{x_1\cdots x_7 = ax_8} \psi(x_1 + \ldots + x_7 - x_8)\lambda_2(x_4x_5(x_6x_7)^{-1})\lambda_2(x_8)\Bigg| \leq 7.
		\] 
		Thus we can complete the above sum to
		\begin{align*}
			M_4(&\SH_{\psi, k, a}) = \lim_{n \rightarrow \infty}\frac{|A_{\psi, k_n}|^{-28}}{|k_n|- 1}\sum_{\chi \in \widehat{k_n^*}}\sum_{\substack{x_{1, j}\dots x_{7, j} = ax_{8, j} \\ y_{1, j} \dots y_{7, j} = ay_{8, j}}} \Bigg(\chi(P(x_{i,j}, y_{i, j})
			% \prod_{j \in \{1, 2\}} x_{4, j}x_{5, j}y_{6, j}y_{7, j}(y_{4, j}y_{5, j}x_{6, j}x_{7, j})^{-1}
			)\\&\times \lambda_2\Bigg(\prod_{j \in \{1, 2\}} x_{8, j}(y_{8, j})^{-1}\Bigg)\psi\Bigg(\sum_{j \in \{1, 2\}} \Bigg(\sum_{i = 1}^7 (x_{i, j} - y_{i, j} )-x_{8, j}+ y_{8, j}\Bigg)\Bigg)\Bigg).
		\end{align*}
		The term $\frac{|A_{\psi, k_n}|^{-28}}{|k_n|- 1}$ is asymptotically equivalent to $|k_n|^{-15}$. 
		We change the order of summation and then the orthogonality of characters of ${k_n^*}$ implies
		\[
		M_4(\SH_{\psi, k, a})  = \lim_{n \rightarrow \infty} f(\psi_n, k_n, a).
		\]
		This is the statement of the theorem.
	\end{proof}
	\begin{corollary}\label{COR_IndOfBaseField}
		Let $k$ be a finite field, $\psi$ a non-trivial additive character of $k$, $a \in k^*$, and $n \geq 1.$ Then
		\[
		M_4(\SH_{\psi, k, a}) = M_4(\SH_{\psi_n, k_n, a}).
		\]
	\end{corollary}
	\begin{proof}
		Theorem \ref{THM_EquiForMom} implies
		\[
		M_4(\SH_{\psi, k, a}) = \lim_{m \rightarrow \infty} f(\psi_m, k_m, a) = \lim_{m \rightarrow \infty} f(\psi_{mn}, k_{mn}, a) = M_4(\SH_{\psi_n, k_n, a}).
		\]
		This is the statement of the corollary.
	\end{proof}
	The following formula for the change of the additive character lies at the heart of our argument.
	\begin{lemma}\label{LEM_ChangeOfPsi}
		Let $k$ be a finite field, $\psi$ a non-trivial additive character of $k$, and $a, \lambda \in k^*$. Define the additive character $\psi_\lambda(x) := \psi(\lambda x)$ for all $x \in k$. Then
		\[
		M_4(\SH_{\psi_\lambda, k, a}) = M_4(\SH_{\psi, k, \lambda^6a}).
		\]
	\end{lemma}
	\begin{proof}
		Just as in the proof of Theorem \ref{THM_EquiForMom} we have
		\[
		M_4(\SH_{\psi_{\lambda}, k, a}) = \lim_{n \rightarrow \infty} \frac{1}{|k_n|- 1}\sum_{\chi \neq \lambda_2} |\text{Tr}(\text{Fr}_{k_n}|H^{\bullet}_c(\BG_{m, \overline{k}}, \SH_{\psi_{\lambda}, k, a}\otimes\mathscr{L}_{\chi}))|^4.
		\]
		We define $\psi_{\lambda, n} := (\psi_{\lambda})_n$. Using Theorem \ref{THM_KatzHypProp}, we can write the above limit in the form
		\[
		M_4(\SH_{\psi_{\lambda}, k, a}) = \lim_{n \rightarrow \infty} \frac{1}{|k_n|- 1}\sum_{\chi \neq \lambda_2} |\text{Tr}\big(\text{Fr}_{k_n}|\CF(\chi, \psi_{\lambda, n}, k_n)_{{a}}\big)|^4.
		\]
		Note that $A_{\psi_\lambda, k} = \lambda_2(\lambda)A_{\psi, k}$. Moreover, substitution implies the formula
		\[
		\sum_{\substack{x_1\cdots x_7 = ax_8\\ x_1, \ldots, x_8 \in k_n^*}} \psi_{\lambda, n} (x_1 + \ldots + x_7 - x_8)\lambda_2(x_8) = \lambda_2(\lambda)\sum_{\substack{y_1\cdots y_7 = \lambda^6ay_8\\y_1, \ldots, y_8 \in k_n^*}} \psi_n(y_1 +\cdots + y_7 - y_8)\lambda_2(y_8).
		\]
		These formulas for changing the character $\psi$ imply
		\[
		|\text{Tr}\big(\text{Fr}_{k_n}|\CF(\chi, \psi_{\lambda, n}, k_n)_{{a}}\big)|^4 = |\text{Tr}\big(\text{Fr}_{k_n}|\CF(\chi, \psi_n, k_n)_{{\lambda^6 a}}\big)|^4
		\]
		for all multiplicative characters $\chi \neq \lambda_2$ of $k_n^*$ by Theorem \ref{THM_KatzHypProp}.  We use this transformation law in the formula for the fourth moment to get
		\begin{align*}
			M_4(\SH_{\psi_{\lambda}, k, a}) = \lim_{n \rightarrow \infty} \frac{1}{|k_n|- 1}\sum_{\chi \neq \lambda_2} |\text{Tr}\big(\text{Fr}_{k_n}|\CF(\chi, \psi_n, k_n)_{{\lambda^6a}}\big)|^4 = M_4(\SH_{\psi, k, \lambda^6a}).
		\end{align*}
		This is what we wanted to prove.
		% Consider the definition
		% \[
		% f(\psi_\lambda, k, a) = |A_{\psi_\lambda, k}|^{-14n}\sum_{\substack{x_1^i\dots x_7^i = x_8^ia \\ y_1^i \dots y_7^i = y_8^ia\\P(x_j^i, y_j^i) = 1}} \psi\Bigg(\lambda\Bigg(\sum_{i, j} (x_j^i - y_j^i) -x_8^i+ y_8^i\Bigg)\Bigg)\lambda_2\Bigg(\prod_i x_8^i(y_8^i)^{-1}\Bigg).
		% \]
		% Make the substitutions $z_i^j = ax_i^j$ and $w_i^j = ay_i^j$ and note that $A_{\psi_\lambda, k} = \lambda_2(\lambda)A_{\psi, k}.$ With these substitutions the sum becomes
		% \[
		% f(\psi_\lambda, k, a) = |A_{\psi, k}|^{-14n}\sum_{\substack{z_1^i\dots z_7^i = z_8^ia\lambda^6\\ w_1^i \dots w_7^i = w_8^ia\lambda^6\\P(z_j^i, w_j^i) = 1}} \psi\Bigg(\sum_{i, j} (z_j^i - w_j^i) -z_8^i+ w_8^i\Bigg)\lambda_2\Bigg(\prod_i z_8^i(w_8^i)^{-1}\Bigg).
		% \]
		% Thus we get the equality
		% \[
		% f(\psi_\lambda, k, a) = f(\psi, k, \lambda^6a).
		% \]
		% Applying Theorem \ref{THM_EquiForMom} twice yields the lemma. 
	\end{proof}
	\begin{remark}
		The equality of moments reflects an isomorphism of sheaves on $\BG_{m, k}^2$. More precisely, this equality follows from a change of characters formula for !-hypergeometric sheaves on $\BG_{m, k}^4$ as in \citep[Def.~8.1.2]{GabberLoeserTore}. For a hypergeometric sheaf on $\BG_{m, k}$, this formula is given by \citep[Lem.~8.7.2]{KatzESDE}.
	\end{remark}
	\section{Weighted Euler-Poincar\'{e} characteristics and the moment} Informally speaking, this section expresses the function $f$ from Definition \ref{DEF_FunctionF} as the trace function of an additive Fourier transform of a sheaf on $\mathbb{A}^1_\BZ$. The equality of the geometric and the arithmetic Tannakian monodromy group implies a formula for the moment in terms of the weight filtration on the stalk of the Fourier transform at 1. This expression, in turn, can be controlled by appealing to the uniformity properties of the Fourier transform. We introduce an ad-hoc notion for a function which is the ``Fourier transform" of a trace function of an $\ell$-adic complex on $\BA^1_{\BZ}$.
	\begin{definition}
		Let $g(\psi, k)$ be a complex-valued function that takes as an input a finite field $k$ and a non-trivial additive character $\psi$ of $k$.  Let $U \subset \text{Spec}(\BZ)$ be a dense open subset. We say that the function $g$ \textit{is representable by a Fourier transform over $U$} if there exists a mixed complex $K$ on $\BA^1_U$ such that 
		\[
		g(\psi, k) = \text{Tr}(\text{Fr}_k|H^{\bullet}_c(\BA^1_{\overline{k}}, K\otimes\SL_{\psi}))
		\]
		for all finite extensions $k/\BF{p}$ with $p \in U$ and all additive characters $\psi$ on $k$. The complex $K$ is said to be a \textit{representing complex for the function $f$.}
	\end{definition}
	\begin{theorem}\label{THM_TraceFormulaForf}
		The function $(\psi, k) \mapsto f(\psi, k, 1)$ is representable by a Fourier transform over $\BZ[1/2\ell]$.
	\end{theorem}
	\begin{proof} The squaring map \([2] : \mathbb{G}_{m,\mathbb{Z}[1/2\ell]} \to \mathbb{G}_{m,\mathbb{Z}[1/2\ell]}\) is a finite étale Galois cover with Galois group \(\mathbb{Z}/2\mathbb{Z}\). Therefore, the pushforward \([2]_\ast \Qbarl\) decomposes as a direct sum \([2]_\ast \Qbarl \cong \Qbarl \oplus \SL\), where \(\SL\) is a non-trivial rank one local system on \(\mathbb{G}_{m,\mathbb{Z}[1/2\ell]}\). By \citep[Cor.~5.3.9]{LeiFu}, for any closed point \(k\) of \(\operatorname{Spec}(\mathbb{Z}[1/2\ell])\), the local system \(\SL\) restricted to the fiber \(\mathbb{G}_{m,k}\) is isomorphic to the direct factor \(\SL_{\lambda_2}\) in the lisse sheaf \([2]_*(\Qbarl)\). Consider the closed subscheme $Z \subset \BG_{m, \BZ[1/2\ell]}^{32}$ defined by the equations 
		\begin{align*}
			x_{1, j}\dots x_{7, j} = x_{8, j},\  y_{1, j} \dots y_{7, j} = y_{8, j},\  P(x_{i, j}, y_{i, j}) = 1
		\end{align*}
		for all $j \in \{1, 2\}.$ Define the maps $\varphi_1\colon Z \rightarrow \BG_{m, \BZ[1/2\ell]}$ and $\varphi_2\colon Z \rightarrow \BA_{\BZ[1/2\ell]}^1$ by
		\begin{align*}
			\varphi_1(x_{i, j}, y_{i, j}) & = \prod_{j \in \{1, 2\}} x_{8 ,i}y_{8 ,i}^{-1} \\
			\varphi_2(x_{i, j}, y_{i, j}) &= \sum_{j \in \{1, 2\}} \bigg(\sum_{i = 1}^7 (x_{i, j} - y_{i, j} )-x_{8, j}+ y_{8, j}\bigg).
		\end{align*}
		Put $K := R\varphi_{2!}(\varphi_1^*(\SL(-15))).$ Consider a finite extension $k/\BF{p}$ with $p$ coprime to $2\ell$ and a non-trivial additive character $\psi$. Denote by $Z_{\overline{k}}$ the fiber of $Z$ over the geometric point $\overline{k}$ in the spectrum of $\BZ[1/2\ell]$. The projection formula implies that there is a Frobenius-equivariant isomorphism
		\[
		H^\bullet_c(\BA^1_{\overline{k}}, K\otimes\SL_{\psi}) \cong H^\bullet_c(Z_{\overline{k}}, \varphi_1^*(\SL_{\lambda_2})\otimes\varphi_2^*(\SL_{\psi}))(-15).
		\]
		The Lefschetz trace formula and the twist in the definition of $K$ imply 
		\[
		\text{Tr}(\text{Fr}_k|H^\bullet_c(Z_{\overline{k}}, \varphi_1^*(\SL_{\lambda_2})\otimes\varphi_2^*(\SL_{\psi}))(-15)) = f(\psi, k, 1).
		\]
		The previous two equations combined imply that $K$ represents the function from the lemma.   
		% This can be proven with the trace formula, the proper base     change theorem, and Deligne's finiteness theorem (see \citep{DeligneSGA45}). The mixedness follows from \citep[Ch. 6]{DeligneWeilII}.
	\end{proof}
	To get the most out of Theorem \ref{THM_TraceFormulaForf}, we use the following well-known proposition. We delay the proof of this proposition to Section \ref{SEC_Lemma} because it is unrelated to the rest of the article.
	\begin{proposition}\label{PROP_SumOfPowers}
		Consider $\lambda_1, \ldots, \lambda_n \in \BC$ and $\alpha_1, \ldots, \alpha_n \in \BC$ such that the limit $\lim_{N \rightarrow \infty} \sum_{i = 1}^n \alpha_i\lambda_i^N$ exists. Then it is given by 
		\[
		\lim_{N \rightarrow \infty} \sum_{i = 1}^n \alpha_i\lambda_i^N = \sum_{|\lambda_i| = 1} \alpha_i.
		\]
	\end{proposition}
	Recall the definition of the weighted Euler-Poincar\'{e} characteristic.
	\begin{definition} Let $w\in \BR$ and $q \in \BN$ a prime power. Consider a $\BC[T]$-module $V$ which is finite-dimensional over $\BC$. Define $V_w \subset V$ to be the sum of all generalized eigenspaces of $T$ acting on $V$ with respect to eigenvalues $\lambda$ whose absolute value satisfies $|\lambda| = q^{w/2}.$
		
		Let $M$ be a bounded complex of $\BC[T]$-modules, whose cohomology groups $H^i(M)$ are finite-dimensional over $\BC$. Define the \textit{weighted Euler-Poincar\'{e} characteristic of $M$ with weight $w$} to be (see \citep[p.~92]{KatzSommeExpon})
		\[
		\chi_w(M) := \sum_{i \in \BZ}(-1)^i \text{dim}_{\BC}(H^i(M)_w).
		\]
	\end{definition}
	\begin{theorem}\label{THM_MomentEulerPoincare}
		Let $k$ be a finite field of characteristic coprime to $2\ell$, $K$ a representing complex for the function $(\psi, k) \mapsto f(\psi, k, 1)$, $a \in k^*$, and $\psi$ a non-trivial additive character of $k$. Then we have
		\[
		M_4(\SH_{\psi, k, a}) = \chi_{0}(H^\bullet_c(\BA^1_{\overline{k}}, K\otimes\SL_{\psi})).
		\]
	\end{theorem}
	\begin{proof}
		Theorem \ref{THM_EquiForMom} and the definition of representability imply
		\begin{equation*}
			M_4(\SH_{\psi, k, a}) = \lim_{n \rightarrow \infty} \textnormal{Tr}(\textnormal{Fr}_k^n|H^\bullet_c(\BA^1_{\overline{k}}, K\otimes\SL_{\psi})).
		\end{equation*}
		For all $i \in \BZ$, let $\lambda_{i, j} \in \BC$ where $1 \leq j \leq d_i$ be the eigenvalues of Frobenius acting on the cohomology group $H^{i}_c(\BA^1_{\overline{k}}, K\otimes\SL_{\psi})$. We have $$\textnormal{Tr}(\textnormal{Fr}_k^n|H^{\bullet}_c(\BA^1_{\overline{k}}, K\otimes\SL_{\psi})) = \sum_{i \in \BZ}\sum_{j = 1}^{d_j} (-1)^i\lambda_{i, j}^n,$$ so Proposition \ref{PROP_SumOfPowers} implies  
		\[
		M_4(\SH_{\psi, k, a})  = \sum_{|\lambda_{i, j}| = 1} (-1)^i = \chi_{0}\big(H^\bullet_c(\BA^1_{\overline{k}}, K\otimes\SL_{\psi})\big).
		\]
	\end{proof}
	\vspace{-8mm}
	\section{Determination of the Tannakian monodromy group} We recall a theorem on the uniformity of the Fourier transform by Katz. For an exposition of uniformity results for the Fourier transform, the reader is referred to \citep[Ch.~V]{KiehlWeissPervSheafFourier}. 
	\begin{theorem}[\hspace*{-0mm}{\citep[p.~92, Cor.~1]{KatzSommeExpon}}] \label{THM_KatzUniformity}
		Consider a ring $R \subset \BC$ that is finitely generated over $\BZ$, a constructible, mixed complex of $\Qbarl$-sheaves $K$ on $\BA^1_R$ and an integer $w \in \BZ$. Then there exists a non-zero $r \in R$ such that for all ring morphisms $R[1/r\ell] \rightarrow k$ into a finite field $k$ and all non-trivial additive characters $\psi$ of $k$, the integer
		\[
		\chi_{w}(H^{\bullet}_{c}(\mathbb{A}_{\overline{k}}^1, K\otimes\SL_{\psi}{})),
		\]
		where the restriction is taken along the map $\BA^1_k \rightarrow \BA^1_R$ induced by the ring morphism, is independent of the ring morphism and the character.
	\end{theorem}
	We have collected all the required results to prove the main theorem.
	\begin{theorem}\label{THM_MainResult}
		There exists a constant $C > 2$ such that for all primes $p >  C$, all finite extensions $k/\BF{p}$, all non-trivial additive characters $\psi$ of $k$, and all $a \in k^*$ the perverse sheaf $\SH_{\psi, k, a}$ has Tannakian monodromy group $G_2$.
	\end{theorem}
	\begin{proof}
		Let $K$ be a representing complex for the function $(\psi, k) \mapsto f(\psi, k, 1)$. Remark that Theorem \ref{THM_KatzUniformity} implies that there is a constant $C > 2$ such that for each prime $p > C$, each finite extension $k/\BF{p}$, and each non-trivial additive character $\psi$ of $k$ the integer
		\begin{equation*}
			\chi_{0}(H^{\bullet}_{c}(\mathbb{A}_{\overline{k}}^1, K\otimes\SL_{\psi}{}))
		\end{equation*}
		does neither depend on $\psi$ nor on $k$. 
		
		Let $p > C$ be a prime, $k/\BF{p}$ a finite extension, $\psi$ a non-trivial additive character, and $a \in k^*$. There exists a finite extension $k_n$ of $k$ and an element $b \in k_n$ such that  $M_4(\SH_{\psi_n, k_n, b}) = 4$ by Theorem \ref{THM_Katz}. Define the field $k_m := k_n(a^{1/6}, b^{1/6})$ and the non-trivial additive character $\psi'(x) := \psi_m(a^{1/6}x)$ of $k_m.$ Corollary \ref{COR_IndOfBaseField}, Lemma \ref{LEM_ChangeOfPsi} and Theorem \ref{THM_MomentEulerPoincare} imply
		\[
		M_4(\SH_{\psi, k, a}) = M_4(\SH_{{\psi'}, k_m, 1}) = \chi_0(H^\bullet_c(\BA_{\overline{k}}^1, K\otimes\SL_{\psi'})).
		\]
		Define the additive character $\psi''(x) := \psi_m(b^{1/6}x)$ of $k_m$. Corollary \ref{COR_IndOfBaseField}, Lemma \ref{LEM_ChangeOfPsi} and Theorem \ref{THM_MomentEulerPoincare} imply
		\[
		4 = M_4(\SH_{\psi_m, k_m, b}) = M_4(\SH_{\psi'', k_m, 1}) =  \chi_0(H^\bullet_c(\BA_{\overline{k}}^1, K\otimes\SL_{\psi''})).
		\]
		The uniformity result says 
		\[
		\chi_0(H^\bullet_c(\BA_{\overline{k}}^1, K\otimes\SL_{\psi''})) = \chi_0(H^\bullet_c(\BA_{\overline{k}}^1, K\otimes\SL_{\psi'})),
		\] so we can chain all the equations to get $$M_4(\SH_{\psi, k, a}) = 4.$$ Hence Theorem \ref{THM_Katz} implies that the Tannakian monodromy group of $\SH_{\psi, k, a}$ is $G_2.$
		
		% Theorem \ref{THM_Katz} states that there is a finite extension $k_n/k$ and $b \in k_n^*$ such that the Tannakian monodromy group of $\SH_{\psi_n, k_n, b}$ is $G_2.$ Let $k_m \cong k_n(b^{1/6}, a^{1/6})$, then Lemma \ref{LEM_ChangeOfPsi}, Theorem \ref{THM_Katz} and Theorem \ref{THM_MomentEulerPoincaré} imply
		% \[
		% 4 = M_4(\SH_{\psi, k_m, b}) = M_4(\SH_{\psi_{b^{1/6}}, k'', 1}) = \chi_{0}(H^{\bullet}_{c}(\mathbb{A}_{\overline{k}}^1, K\otimes\SL_{\psi_{b^{1/6}}}{})) = n.
		% \]
		% Thus wet get $n = 4.$ Since the characteristic of $k$ is large enough Equation (\ref{EQN_Uniformity}) and Theorem \ref{THM_MomentEulerPoincaré} give
		% \[
		% 4  = \chi_{0}(H^{\bullet}_{c}(\mathbb{A}_{\overline{k}}^1, K\otimes\SL_{\psi_{a^{1/6}}}{})) = M_4(\SH_{\psi_{a^{1/6}}, k'', 1}) = M_4(\SH_{\psi, k, a}). 
		% \]
		% Theorem \ref{THM_Katz} implies that the Tannakian monodromy group of $\SH_{\psi, k, a}$ is $G_2.$
	\end{proof}
	\section{Sums of powers}\label{SEC_Lemma} In this section, we prove the remaining Proposition \ref{PROP_SumOfPowers}. This proposition and the following lemmas are well-known but we prove them here because we were not able to find a reference. Define $$\BT := \{z \in \BC : |z| = 1\}.$$
	% \begin{lemma}
		%     Consider $a_1, \ldots, a_n \in \BR$, which are $\BQ$-linearly independent. The subset
		%     \[
		%     \{(e(ma_1), \ldots, e(ma_n)) : m \in \BZ \} \subset \BT^n
		%     \]
		%     is dense.
		% \end{lemma}
	
	\begin{lemma}\label{LEM_SumOfPowersPrep}
		Let $x \in \BT^n$ and define $A \subset \BT^n$ to be the closure of the subgroup generated by $x.$ A continuous function $f\colon A \rightarrow \BC$ such that
		\[
		\lim_{N \rightarrow \infty} f(x^N)
		\]
		exists is constant.
	\end{lemma}
	This is a consequence of the Kronecker--Weyl theorem, see for example \citep[Thm.~B.6.5~(1)]{KowalskiProbNumberTheory}. The Kronecker--Weyl theorem implies that the set $\{x^n : n \geq N\}$ is everywhere dense in $A$ for any $N \geq 0$ (since the indicator function of a non-empty open subset can never have measure zero with respect to a Haar measure), thereby proving Lemma \ref{LEM_SumOfPowersPrep}.
	% \begin{proof} Let $N \geq 0.$ The group $A$ is the closure of the set $\{t^n : n \in \BZ\}$. Thus any non-trivial continuous character $\nu\colon A \rightarrow \BT$ satisfies $\nu(t) \neq 1$. This implies
		% \[
		% \frac{1}{N - M}\sum_{n = N}^{M - 1} \nu(t)^{n} = \frac{1}{M - N} \frac{\nu(Mt) - \nu(Nt)}{\nu(t) - 1}
		% \]
		% for all $M > N.$ When $M \rightarrow \infty$, this expression converges to zero. The Peter-Weyl theorem implies that the sets $\{x^n: M - 1 \geq n \geq N \}$ equidistribute in $A$ as $M \rightarrow \infty$.
		
		% Define $\alpha := \lim_{n \rightarrow \infty} f(x^n)$ and consider $\epsilon > 0$. There exists $N \geq 0$ such that
		% \[
		% |f(x^n) - \alpha| < \epsilon
		% \]
		% for all $n \geq N.$ Thus
		% \[
		% \int_A |f(a) - \alpha| da = \lim_{M \rightarrow \infty} \frac{1}{M - N} \sum_{n = N }^{M - 1}|f(x^n) - \alpha| \leq \epsilon.
		% \]
		% Since $\epsilon$ was arbitrary, we get
		% \[
		% \int_A |f(a) - \alpha| da = 0.
		% \]
		% The continuity of $f$ implies $f(a) = \alpha$ for all $a \in A.$
		%      % We argue by contradiction. By reordering the numbers, we can assume $|\lambda_1| = 1$ and $\lambda_1 \neq 1$. There are subsequences of $(\lambda_1^k, \ldots, \lambda_n^k)_k$ which converge to $(1, w_2, \ldots, w_n)$ and to $(z, w_2, \ldots, w_n)$ for some $z \neq 1$. However, this contradicts the existence of the limit. 
		% \end{proof}
	\begin{lemma}\label{LEM_SumsOfPowersPrep2}
		Define the functions $\phi_z \colon \BN \rightarrow \BC$
		\[
		\phi_z(n) = z^n
		\]
		for each $z \in \BC.$ These functions are linearly independent in $\BC^\BN$.
	\end{lemma}
	\begin{proof}
		% Define $S := \{z \in \BC: \alpha_z \neq 0 \lor \beta_z \neq 0\}$ and $n := |S|$,
		Suppose
		\[
		\lambda_0\phi_{z_0} + \ldots + \lambda_m\phi_{z_m} = 0
		\]
		for pairwise distinct $z_i \in \BC$ and arbitrary $\lambda_i \in \BC$. 
		Note that the Vandermonde matrix
		\[
		\begin{pmatrix}
			1 & 1  &\cdots & 1 \\
			z_0 & z_1 & \cdots & z_m \\
			\vdots & \vdots & \vdots & \vdots \\
			z_0^m & z_1^m & \cdots & z_m^m.
		\end{pmatrix}
		\]
		% the vector $\underline{\beta} := (\beta_z)_{z \in S}$, and the vector $\underline{\alpha} := (\alpha_z)_{z \in S}$. The assumption says
		% \[
		% A^t\underline{\alpha} = A^t\underline{\beta}.
		% \]
		% The Vandermonde matrix
		is invertible because of the well-known formula for the determinant. The linear relation implies
		\[
		\sum_{i = 0}^m \lambda_i
		\begin{pmatrix}
			1 \\ z_i \\ \vdots \\ z_i^m
		\end{pmatrix} = 0.
		\]
		The invertibility of the Vandermonde matrix implies $\lambda_i = 0$ which implies that the $\phi_z$ are linearly independent.
	\end{proof}
	We prove Proposition \ref{PROP_SumOfPowers}. Suppose $|\lambda_i| \leq 1.$ We have
	\[
	\lim_{N \rightarrow \infty} \sum_i \alpha_i\lambda_i^N = \lim_{N \rightarrow \infty} \sum_{\substack{|\lambda_i| = 1}} \alpha_i\lambda_i^N
	\]
	because the terms with $|\lambda_i| < 1$ converge to zero as $N \rightarrow \infty.$ Let $\mathbf{\lambda} \in \BT^n$ be the tuple of all $\lambda_i$ with $|\lambda_i| = 1$ and consider the subgroup $A$ generated by $\mathbf{\lambda}$. The function
	\[
	f\colon (\gamma_i)_i \in A \mapsto \sum_{|\lambda_i| = 1} \alpha_i\gamma_i
	\]
	is continuous. The existence of the limit $\lim_{N \rightarrow \infty} f(\mathbf{\gamma}^N)$ implies that the function is constant by Lemma \ref{LEM_SumOfPowersPrep}. Hence
	\[
	\lim_{N \rightarrow \infty} \sum_{\substack{|\lambda_i| = 1}} \alpha_i\lambda_i^N = \lim_{N \rightarrow \infty} f(\mathbf{\gamma}^N) = \lim_{N \rightarrow \infty} f(\mathbf{\gamma}^0) = f(\mathbf{\gamma}^0) = \sum_{\substack{|\lambda_i| = 1}} \alpha_i\lambda_i^0 = \sum_{|\lambda_i| = 1} \alpha_i.
	\]
	
	Consider arbitrary $\lambda_i \in \BC$ and define 
	\[
	\beta_z := \sum_{\lambda_i = z} \alpha_i
	\]
	for each $z \in \BC$. We can write
	\[
	\sum_i \alpha_i \lambda_i^N = \sum_{z \in \BC} \beta_z z^N 
	\]
	for each $N \geq 0.$ Let $M := \max(\{|z| : \beta_z \neq 0\}\cup\{1\})$ and suppose $M > 1.$ Then
	\[
	\lim_{N \rightarrow \infty} \sum_{z \in \BC} \beta_z (z/M)^N = \Big(\lim_{N \rightarrow \infty} \sum_{z \in \BC} \beta_z z^N\Big)\cdot\Big(\lim_{N \rightarrow \infty} M^{-N}\Big) = 0.
	\]
	Thus
	\[
	0 = \lim_{N \rightarrow \infty} \sum_{|z| = M} \beta_z (z/M)^N.
	\]
	The powers all have absolute value $\leq 1$, by definition of $M$, so the above argument implies
	\[
	0 = \sum_{|z| = M} \beta_z (z/M)^N
	\]
	for all $N \geq 0$. Lemma \ref{LEM_SumsOfPowersPrep2} implies $\beta_z = 0$ for all $z \in \BC$ with $|z| = M$. This is a contradiction to the definition of $M$, and therefore $M \leq 1.$ Thus we are in the special case treated above.
	\hfill \qed
	\section*{Funding} \noindent This work was supported by the SNF grant 219220.
	\section*{Acknowledgments} \noindent The author would like to thank Prof. Dr. Emmanuel Kowalski for suggesting this interesting problem and his continued interest and support. The author expresses deep gratitude towards the referee for their many suggestions. They have improved the article substantially.

\end{document}